\documentclass[a4paper,12pt]{article}
\usepackage[utf8]{inputenc}
\usepackage{amsmath}
\usepackage{amssymb}
\usepackage{amsthm}
\usepackage{cite}

\title{Singularly perturbed reaction-diffusion problems as first order systems}
\author{Sebastian Franz\footnote{
          Institute of Scientific Computing, Technische Universit\"at Dresden, Germany.
          \mbox{e-mail}: sebastian.franz@tu-dresden.de}
       }
\date{\today}

\makeatletter
\renewcommand*\env@matrix[1][r]{\hskip -\arraycolsep
  \let\@ifnextchar\new@ifnextchar
  \array{*\c@MaxMatrixCols #1}}
\makeatother
\DeclareMathOperator{\Grad}{{grad}}
\DeclareMathOperator{\Div}{{div}}
\DeclareMathOperator{\Curl}{{curl}}
\DeclareMathOperator{\meas}{meas}
\newcommand{\e}{\mathrm{e}}
\newcommand{\pt}{\partial}
\newcommand{\eps}{\varepsilon}
\newcommand{\laplace}{\Delta}
\newcommand{\norm}[2]{\|{#1}\|_{#2}}
\newcommand{\tnorm}[1]{\left|\!\!\;\left|\!\!\;\left| {#1}
                       \right|\!\!\;\right|\!\!\;\right|}
\newcommand{\I}{\mathcal{I}}
\newcommand{\R}{\mathbb{R}}
\newcommand{\U}{\mathcal{U}}
\newcommand{\PS}{\mathcal{P}}
\newcommand{\QS}{\mathcal{Q}}
\newcommand{\vecsymb}[1]{\boldsymbol{#1}}
\newcommand{\vn}{\vecsymb{n}}
\newcommand{\vq}{\vecsymb{q}}
\newcommand{\vs}{\vecsymb{s}}
\newcommand{\vu}{\vecsymb{u}}
\newcommand{\vv}{\vecsymb{v}}
\newcommand{\vw}{\vecsymb{w}}
\newcommand{\vI}{\boldsymbol{\mathcal{I}}}
\newcommand{\vJ}{\boldsymbol{\mathcal{J}}}
\newcommand{\vP}{\boldsymbol{\mathcal{P}}}
\newcommand{\valpha}{\vecsymb{\alpha}}
\newcommand{\veta}{\vecsymb{\eta}}
\newcommand{\vxi}{\vecsymb{\xi}}
\newcommand{\DS}{\mathcal{D}}
\newcommand{\scp}[1]{\left\langle #1 \right\rangle}
\newcommand{\pmtrx}[1]{\ensuremath{\begin{pmatrix}#1 \end{pmatrix}}}
\newcommand{\rarrow}{\quad\Rightarrow\quad}
\newcommand{\qmbox}[1]{\quad\mbox{#1}\quad}

\theoremstyle{plain}
\newtheorem{theorem}{Theorem}[section]
\newtheorem{lemma}[theorem]{Lemma}
\newtheorem{ass}[theorem]{Assumption}
\newtheorem{corollary}[theorem]{Corollary}
\newtheorem{remark}[theorem]{Remark}
\begin{document}
  \maketitle
  \begin{abstract}
    We consider a singularly perturbed reaction diffusion problem as a first order two-by-two system.
    Using piecewise discontinuous polynomials for the first component and $H_{\Div}$-conforming 
    elements for the second component we provide a convergence analysis on layer adapted meshes 
    and an optimal convergence order in a balanced norm that is comparable with a balanced 
    $H^2$-norm for the second order formulation.
  \end{abstract}

  \textit{AMS subject classification (2010): 65N12, 65N15, 65N30} 

  \textit{Key words: reaction diffusion problem, singularly perturbed, first order system, balanced norm} 

  \section{Introduction}\label{sec:intro}
  Consider the singularly perturbed reaction diffusion problem, given in $\Omega=(0,1)^2$ by
  \begin{gather}\label{eq:original}
    -\eps^2\laplace u+cu=f,
  \end{gather}
  where $0<\eps\ll1$, $c\in W^{1,\infty}$, $c_\infty\geq c\geq c_0>0$ and $u=0$ on $\pt\Omega$. 
  We rewrite the problem, using $\vu=-\eps\Grad^\circ u$, into a first order system
  \begin{gather}\label{eq:fos}
    \left[\pmtrx{c&0\\0&1}+\pmtrx{[c]0&\eps\Div\\\eps\Grad^\circ&0}\right]\pmtrx{u\\\vu}=\pmtrx{f\\0},
  \end{gather}
  where $\Grad^\circ$ denotes the gradient in $H^1_0(\Omega)$ and $\Div$ its adjoint, the divergence.  
  This formulation is also called a mixed formulation. For its weak formulation let $\scp{\cdot,\cdot}$ denote 
  the $L^2$-scalar product over $\Omega$. Then \eqref{eq:fos} becomes with $V=(v,\vv)\in L^2(\Omega)\times H_{\Div}(\Omega)$
  \begin{gather*}
    \scp{cu,v}+\eps\scp{\Div\vu,v}+\scp{\vu,\vv}+\eps\scp{\Grad^\circ u,\vv}=\scp{f,v}
  \end{gather*}
  which can also be written for $U=(u,\vu)\in L^2(\Omega)\times H_{\Div}(\Omega)$ as
  \begin{gather}\label{eq:weak}
    \scp{cu,v}+\eps\scp{\Div\vu,v}+\scp{\vu,\vv}-\eps\scp{u,\Div\vv}=\scp{f,v}.
  \end{gather}
  This is the weak form we will discretise and analyse.
  
  Singularly perturbed reaction diffusion problems were analysed in many papers, see e.g. \cite{Apel99,RST08}
  The associated norm to \eqref{eq:original} is the $\eps$-weighted $H^1$-norm, also called energy norm. 
  Unfortunately, that norm is not strong enough to see the boundary layers. 
  For the boundary $x=0$ the corresponding layer function is of the type $\e^{-x/\eps}$. Here it holds
  \[
    \norm{\e^{-x/\eps}}{L^2(\Omega)}+\eps\norm{\Grad\e^{-x/\eps}}{L^2(\Omega)}
    \lesssim \eps^{1/2}
    \stackrel{\eps\to 0}{\longrightarrow}0.
  \]
  Therefore over the last years convergence in a balanced norm, where the boundary 
  layers do not vanish for $\eps\to 0$, was considered, see \cite{LS11,RSch15,FrR19,AMM19}. 
  For the lowest order Raviart-Thomas elements on a Shishkin mesh the system \eqref{eq:weak} 
  was also considered in \cite[Section 5]{Roos16} and analysed in a balanced $H^1$-comparable norm.
  
  In this paper we prove optimal convergence orders in a stronger balanced $H^2$-comparable norm 
  \[
    \tnorm{U}_{bal}^2\sim\norm{u}{L^2(\Omega)}^2+\eps^{-1}\norm{\vu}{L^2(\Omega)}^2+\eps\norm{\Div \vu}{L^2(\Omega)}^2
  \]
  for a variety of $H_{\Div}$-conforming 
  elements on general layer-adapted meshes. The paper is organised as follows. In Section~\ref{sec:basics} we define the numerical method and recall 
  results for a solution decomposition and interpolation errors. In Section~\ref{sec:analysis} we provide
  the convergence analysis and in the final Section~\ref{sec:numerics} some numerical examples illustrating our
  theoretical results are given.
  
  \textbf{Notation:} We denote vector valued functions with a bold font. $L^p(D)$ with the norm $\norm{\cdot}{L^p(D)}$
  is the classical Lebesque space of function integrable to the power $p$ over a domain $D\subset\R^2$ and $W^{\ell,p}(D)$
  the corresponding Sobolev space for derivatives up to order $\ell$. Furthermore, we write $A\lesssim B$ if there exists 
  a generic constant $C>0$ such that $A\leq C\cdot B$.
  
  \section{Numerical method and interpolation errors}\label{sec:basics}
  In order to define our numerical method, we need discrete spaces defined over an appropriate mesh. 
  A basic tool for defining this mesh is the knowledge of a solution decomposition, especially
  the structure of layers. 
  \begin{ass}
    The solution $u$ of \eqref{eq:original} can be written as
    \[
      u=s+w_1+w_2+w_3+w_4+w_{12}+w_{23}+w_{34}+w_{41}
    \]
    where $s$ is the smooth part, $w_i$ are boundary layers and $w_{ij}$ are corner layers (both counted counterclockwise).
    To be more precise, for any given degree $k$ it holds for $0\leq i,j\leq k+2$,
    \begin{align*}
      \norm{\pt_x^i\pt_y^j s}{L^\infty(\Omega)}&\lesssim 1,&
      |\pt_x^i\pt_y^jw_1(x,y)|&\lesssim \eps^{-i}\e^{-x/\eps},&
      |\pt_x^i\pt_y^j w_{12}(x,y)|&\lesssim \eps^{-(i+j)}\e^{-(x+y)/\eps},
    \end{align*}
    and analogously for the other boundary layers and corner layers.
  \end{ass}
  \begin{remark}
    Using above solution decomposition for $u$ we derive a similar decomposition
    for the solution $U$ of \eqref{eq:weak}, as $U=(u,\vu)$ and $\vu=-\eps\Grad u$.
    
    Such assumptions on a solution decomposition are very common in the analysis of singularly perturbed
    problems. They hold true under compatibility and regularity conditions on the data, see e.g. \cite{LMSZ08,CGOR05,HK90}. 
  \end{remark}

  We follow \cite{RL99} and construct an S-type mesh using the information of the solution decomposition.
  First, we define a transition point $\lambda$, such that a typical boundary layer function is small enough:
  \[
    \exp(-\lambda/\eps) = N^{-\sigma}
    \rarrow
    \lambda = \sigma\eps\ln(N),
  \]
  for a constant $\sigma>0$ specified later. We additionally assume
  \[
    \lambda = \min\left\{\sigma\eps\ln(N),\frac{1}{4}\right\},    
  \]
  as otherwise $\eps$ is large enough to facilitate a standard numerical analysis.
  Now $0=x_0<x_1<\dots<x_N=1$ are given by
  \[
   x_i:=\begin{cases}
          \sigma\eps\phi\left(\frac{2i}{N}\right), &i=0,\dots,N/4,\\
          \frac{2i}{N}(1-2\lambda)-\frac{1}{2}+2\lambda,          &i=N/4,\dots,3N/4,\\
          1-\sigma\eps\phi\left(2-\frac{2i}{N}\right), &i=3N/4,\dots,N,\\
        \end{cases}
  \]
  where $\phi$ is a mesh-generating function with the properties
  \begin{itemize}
    \item $\phi$ is monotonically increasing,
    \item $\phi(0)=0$ and $\phi(1/2)=\ln N$,
    \item $\phi$ is piecewise differentiable with $\max \phi'\leq C N$ and
    \item $\min\limits_{i=1,\dots,N/4}\left(\phi\left(\frac{2i}{N}\right)-\phi\left(\frac{2(i-1)}{N}\right)\right)\geq C N^{-1}$.
  \end{itemize}
  The first three conditions are given in \cite{RL99}, while the last one allows the mesh-widths
  inside the boundary layers to be bounded from below, see also \cite{FrM10_1}. Related to $\phi$ we define the mesh 
  characterising function $\psi$ by
  \[
    \psi=\e^{-\phi}.
  \]
  Several S-type meshes are given in \cite{RL99} fulfilling above properties. We only provide the definitions of the 
  two mostly used. For the Shishkin mesh we have
  \[
    \phi(t) = 2t\ln N,\quad
    \psi(t) = N^{-2t},\quad
    \max|\psi'| =2\ln N
  \]
  and the Bakhvalov-S-mesh
  \[
    \phi(t)=-\ln(1-2t(1-N^{-1})),\quad
    \psi(t)=1-2t(1-N^{-1}),\quad
    \max|\psi'|= 2.
  \]
  In addition to the mesh generating and characterising functions also $\max|\psi'|:=\max\limits_{t\in[0,1/2]}|\psi'(t)|$
  is given, that enters all the error estimates on S-type meshes.
  
  The two-dimensional mesh $T_N$ is then defined by all cells $K_{ij}:=(x_{i-1},x_i)\times(x_{j-1},x_j)$ for $1\leq i,j\leq N$.
  Note that it holds
  \begin{gather}\label{eq:hbound}
    h_i:=x_i-x_{i-1}
       \lesssim \begin{cases}
                  \eps N^{-1}\max|\psi'|e^{x/(\sigma\eps)},&i\leq N/4\text{ or }i>3N/4,\text{ and }x\in[x_{i-1},x_i],\\
                  N^{-1},& \text{otherwise},
                \end{cases}
  \end{gather}
  and also the simpler bound
  \[
    h:=\max_{i=1,\dots,N/4}h_i
      \lesssim \eps.
  \]
  Let us denote two subdomains of $\Omega$ per layer function, exemplarily given for $w_1$ by
  \[
    \Omega_1:=[0,\lambda]\times[0,1]\qmbox{ and }\Omega_1^*:=[0,x_{N/4-1}]\times[0,1]\subset\Omega_1
  \]
  and for $w_{12}$ by 
  \[
    \Omega_{12}:=[0,\lambda]^2\qmbox{ and }\Omega_{12}^*:=[0,x_{N/4-1}]^2\subset\Omega_{12}.
  \]  
  With \eqref{eq:weak} only needing $L^2$-regularity for the first component and $H_{\Div}$-regularity for the second component,
  our discrete spaces are
  \[
    \U_N:=\{(u_N,\vu_N)\subset L^2(\Omega)\times H_{\Div}(\Omega): \forall K\subset T_N: u_N|_K\in\QS_k(K),\,\vu_N|_K\in \DS_k(K)\},
  \]
  where $\QS_k(K)$ is the space of polynomials with degree up to $k$ in each variable on the cell $K$ of $T_N$.
  For the discretisation of $H_{\Div}$ with $\DS_k(K)$ we can use 
  \begin{itemize}
    \item the Raviart-Thomas space
          \[
            RT_k(K)=\QS_{k+1,k}(K)\times\QS_{k,k+1}(K),
          \]
          introduced by Raviart and Thomas in \cite{RT77} on triangular meshes, see also \cite{BBF13} for rectangular meshes,
          where $\QS_{p,q}(K)$ is the space of polynomials with degree $p$ in $x$ and degree $q$ in $y$ on the cell $K$ or
    \item the Brezzi-Douglas-Marini space
          \[
            BDM_k( K):= (\PS_k(K))^2\oplus\text{span}\{\Curl (x^{k+1}y) ,\Curl (xy^{k+1})\},
          \]
          see \cite{BDM85}, where $\PS_k(K)$ is the space of polynomials of total degree $k$ on the cell $K$
          and $\Curl w=(\pt_y w,-\pt_x w)$.
  \end{itemize}

  Then the discrete method reads: Find $U_N=(u_N,\vu_N)\in \U_N$, s.t. for all $V\in\U_N$ it holds
  \begin{gather}\label{eq:blf}
    B(U_N,V):=\scp{cu_N,v}+\eps\scp{\Div\vu_N,v}+\scp{\vu_N,\vv}-\eps\scp{u_N,\Div\vv}=\scp{f,v}.
  \end{gather}
  Note that the solution $U$ of \eqref{eq:weak} does also fulfill \eqref{eq:blf} and we therefore have
  Galerkin orthogonality
  \begin{gather}\label{eq:GalOrth}
    B(U_N-U,V)=0,\quad \forall V\in\U_N.
  \end{gather}
  
  \section{Numerical analysis}\label{sec:analysis}
  Let us start with an interpolation operator into $\U_N$ given by its two components. 
  The first one $\I_1$ will be a weighted local $L^2$-projection, defined on any $K\subset T_N$
  by
  \[
    \scp{c(\I_1 u-u),v}_K=0\text{ for all } v\in \QS_k(K).
  \]
  Note that we have similar to the standard $L^2$-projection anisotropic interpolation error estimates, 
  i.e. for any $0\leq\ell\leq k+1$ it holds on a cell $K$ with dimension $h_x\times h_y$
  \begin{gather}\label{eq:aniso:L2}
    \norm{\I_1v-v}{L^2(K)}\lesssim h_x^\ell\norm{\pt_x^\ell v}{L^2(K)}+h_y^\ell\norm{\pt_y^\ell v}{L^2(K)}
  \end{gather}
  if $v\in H^{\ell}(K)$. 
  These estimates can be shown with the help of a pointwise interpolation operator and 
  the $L^2$-stability of the $L^2$-projection, see \cite{RSch15}, or directly by using the theory of Apel, \cite{Apel99}.
  Note that the $L^2$-projection is also $L^\infty$-stable, see also \cite{Oswald13}.
  
  The second operator $\vI_2$ utilises the classical interpolation operator $\vJ$ on $\DS_k$. 
  It is defined on each cell $K$ for
  \begin{itemize}
    \item $\DS_k(K)=RT_k(K)$ by
          \begin{subequations}\label{eq:inter:RT}
          \begin{align}
            \int_{F}(\vJ\vv-\vv)\cdot \vn\cdot q&=0,\quad\forall q\in \PS_k(F)\text{ for all faces } F\subset\partial K,\\
            \int_{K}(\vJ\vv-\vv)\cdot\vq&=0,\quad\forall \vq\in \QS_{k-1,k}(K)\times\QS_{k,k-1}(K),
          \end{align}
          \end{subequations}
    \item $\DS_k(K)=BDM_k(K)$ by
          \begin{subequations}\label{eq:inter:BDM}
          \begin{align}
            \int_{F}(\vJ\vv-\vv)\cdot \vn\cdot q&=0,\quad\forall q\in \PS_{k}(F),\forall F\subset\partial K,\\
            \int_{K}(\vJ\vv-\vv)\cdot\vq&=0,\quad\forall \vq\in (\PS_{k-2}(K))^2.
          \end{align}
          \end{subequations}
  \end{itemize}

  It holds the anisotropic interpolation error estimates for $\vJ$, see \cite{Stynes14,Fr21},  
  \begin{gather}\label{eq:aniso:RT:1}
    \norm{\vJ\vv-\vv}{L^2(K)}\lesssim \sum_{s=0}^{k+1} h_x^{\ell-s}h_y^s\norm{\pt_x^{k+1-s}\pt_y^s \vv}{L^2(K)}
  \end{gather}
  if $\vv\in H^{k+1}(K)$. Note that for $RT_k$  
  an even sharper result involving only pure derivatives of $\vv$ holds, see \cite{Fr21}.
  In addition, we have also anisotropic interpolation error estimates for the $L^2$-norm of the divergence, see \cite{Fr21},
  \begin{subequations}\label{eq:aniso:RT:2}
  \begin{align}
    RT_k:\quad\norm{\Div(\vJ\vv-\vv)}{L^2(K)}
      &\lesssim h_x^{k+1}\norm{\pt_x^{k+1}\Div \vv}{L^2(K)}+h_y^{k+1}\norm{\pt_y^{k+1}\Div \vv}{L^2(K)},\\
    BDM_k:\quad\norm{\Div(\vJ\vv-\vv)}{L^2(K)}
      &\lesssim \sum_{|\valpha|=k}h_x^{\alpha_1}h_y^{\alpha_2}\norm{\pt_x^{\alpha_1}\pt_y^{\alpha_2}\Div \vv}{L^2(K)},
  \end{align}
  \end{subequations}
  if $\vv$ is such that $\Div\vv\in H^{k+1}(K)$ for $RT_k$ and $\Div\vv\in H^{k}(K)$ for $BDM_k$.
  
  If we only want to prove convergence in the $L^2$-norm of $U=(u,-\eps\Grad u)$ the interpolation operator $\vJ$ is enough.
  For a stronger convergence result we define a more sophisticated operator, following ideas from \cite{LS11}.
  Recalling the decomposition of $u$, we have for $\vu=-\eps\Grad u$ the decomposition
  \begin{gather}\label{eq:solDecU2}
    \vu=\vs+\vw^1+\vw^2+\vw^3+\vw^4+\vw^{12}+\vw^{23}+\vw^{34}+\vw^{41}
  \end{gather}
  with the obvious definition of the bold font letters. Now the operator $\vI_2$ is defined piecewise
  on each cell $K\subset T_N$ with $id\in\{1,2,3,4,12,23,34,41\}$
  \begin{align*}
    \vI_2 \vs|_K &:= \vJ \vs|_K,&
    \vI_2 \vw^{id}|_K&:=\begin{cases}
                  \vJ \vw^{id}|_K,& K\subset \Omega_{id}^*,\\
                  \hat\vJ^{id} \vw^{id}|_K,&K\subset\Omega_{id}\setminus\Omega_{id}^*,\\
                  0, & K\subset\Omega\setminus\Omega_{id}.
                \end{cases}
  \end{align*}
  Using $\Gamma_{id}:=\pt\Omega_{id}\setminus\Gamma$
  we define the remaining operators on each $K\subset\Omega_{id}\setminus\Omega_{id}^*$
  using the same definition as for $\vJ$ with the exception of the first condition in \eqref{eq:inter:RT}
  and \eqref{eq:inter:BDM}. 
  This one is replaced by the two conditions
  \begin{align*}
    \int_{F}(\hat\vJ^{id} \vw^{id}-\vw^{id})\cdot\vn\cdot q&=0\text{ for each face }F\subset\partial K\setminus\Gamma_{id},\,\forall q\in\PS_k(F),\\
    \int_{F}\hat\vJ^{id} \vw^{id}\cdot\vn\cdot q&=0\text{ for each face }F\subset\partial K\cap\Gamma_{id},\,\forall q\in\PS_k(F).
  \end{align*}
  For our analysis let us define a norm that is associated with $B(\cdot,\cdot)$. Here it holds
  \begin{gather}\label{eq:coercive}
    B(U,U)\geq\min\{1,c_0\}\norm{U}{L^2(\Omega)}^2
  \end{gather}
  that is equivalent to coercivity in the energy norm of the weak formulation of \eqref{eq:original}. But we can actually use the stronger norm
  \[
    \tnorm{U}:=\left( \norm{U}{L^2(\Omega)}^2+\delta\norm{\eps\Div \vu}{L^2(\Omega)}^2 \right)^{1/2},
  \]
  where $\delta\leq \frac{c_0}{c_\infty^2}$. This norm is equivalent to the weighted $H^2$-norm
  $
    \norm{u}{L^2(\Omega)}+\eps\norm{\Grad u}{L^2(\Omega)}+\eps^2\norm{\laplace u}{L^2(\Omega)},
  $
  which is stronger than the energy norm and, unfortunately, also not balanced. To repair this weakness we also introduce 
  a balanced version of this norm
  \[
    \tnorm{U}_{bal}:=\left( \norm{u}{L^2(\Omega)}^2+\eps^{-1}\norm{\vu}{L^2(\Omega)}^2+\delta\eps^{-1}\norm{\eps\Div \vu}{L^2(\Omega)}^2 \right)^{1/2}.
  \]
  The remainder of this section is devoted to proving optimal uniform convergence orders in the balanced norm. 
  Of course, convergence in the unbalanced norm then follows.
  \begin{lemma}\label{lem:infsup}
    For each $\delta \leq \frac{c_0}{c_\infty^2}$ exists a constant $\beta>0$, such that for all $V\in\U_N$ it holds
    \[
      \sup_{\chi\in\U_N}\frac{B(V,\chi)}{\norm{\chi}{L^2(\Omega)}}\geq \beta\tnorm{V}.
    \]
  \end{lemma}
  \begin{proof}
    By \eqref{eq:coercive} we already have
    \[
      B(V,V)\geq \min\{1,c_0\}\norm{V}{L^2(\Omega)}^2.
    \]
    Choosing as test function $\chi(V)=(v+\delta\eps\Div\vv,\vv)\in\U_N$, we obtain
    \begin{align*}
      B(V,\chi(V))
        &\geq \left(c_0-\delta\frac{c_\infty^2}{2}\right)\norm{v}{L^2(\Omega)}^2
              +\norm{\vv}{L^2(\Omega)}^2
              +\frac{\delta}{2}\norm{\eps\Div\vv}{L^2(\Omega)}^2
    \end{align*}
    and together with $\delta\leq\frac{c_0}{c_\infty^2}$ we have
    \[
      B(V,\chi(V))\geq \frac{c_0}{2}\norm{v}{L^2(\Omega)}^2
                  +\norm{\vv}{L^2(\Omega)}^2
                  +\frac{\delta}{2}\norm{\eps\Div\vv}{L^2(\Omega)}^2
              \geq\frac{1}{2}\min\{c_0,2\}\tnorm{V}^2.
    \]
    In addition it holds
    \begin{align*}
      \norm{\chi(V)}{L^2(\Omega)}^2
      &\leq 2\norm{V}{L^2(\Omega)}^2+2\delta^2\norm{\eps\Div\vv}{L^2(\Omega)}^2
       \leq 2\max\left\{1,\delta\right\}\tnorm{V}^2.
    \end{align*}
    Thus it follows
    \[
      \sup_{\chi\in\U_N}\frac{B(V,\chi)}{\norm{\chi}{L^2(\Omega)}}
      \geq \frac{B(V,\chi(V))}{\norm{\chi(V)}{L^2(\Omega)}}
      \geq \frac{\sqrt{2}}{4}\frac{\min\{2,c_0\}}{\max\{1,\sqrt{\delta}\}}\tnorm{V}.
    \]
    Setting $\beta=\frac{\sqrt{2}}{4}\frac{\min\{2,c_0\}}{\max\{1,\sqrt{\delta}\}}
                  \geq\frac{\sqrt{2}}{4}\frac{\min\{2,c_0\}}{\max\{1,\frac{\sqrt{c_0}}{c_\infty}\}}
                  >0$ proves the assertion.
  \end{proof}
  Let us split the error $U-U_N$ into an interpolation error and a discrete error
  \[
    U-U_N = U - \I U+\I U-U_N=:(\eta,\veta)-(\xi,\vxi),\,(\xi,\vxi)\in \U_N.
  \]
  Using above inf-sup inequality and the Galerkin orthogonality \eqref{eq:GalOrth} we arrive at
  \begin{gather}\label{eq:infsup}
    \beta\tnorm{(\xi,\vxi)}
      \leq \sup_{V\in\U_N}\frac{B((\xi,\vxi),V)}{\norm{V}{L^2(\Omega)}}
      = \sup_{V\in\U_N}\frac{B((\eta,\veta),V)}{\norm{V}{L^2(\Omega)}},
  \end{gather}
  and we are left with estimating $B((\eta,\veta),V)$ for any $V\in\U_N$. Here it holds
  using \eqref{eq:blf}
  \begin{align*}
    B((\eta,\veta),V)
      &= \scp{c\eta,v}+\eps\scp{\Div\veta,v}+\scp{\veta,\vv}-\eps\scp{\eta,\Div\vv}\\
      &= \eps\scp{\Div\veta,v}+\scp{\veta,\vv}-\eps\scp{\eta,\Div\vv}
  \end{align*}
  due to $\I_1$ being the weighted $L^2$-projection. Note that in the case of constant $c$, 
  the last term would also vanish due to $\Div\vv|_K\in\QS_k(K)$.
  \begin{lemma}\label{lem:est:eta2}
    It holds for $\sigma> k+1$
    \[
      \norm{\veta}{L^2(\Omega)}\lesssim \eps^{1/2}(h+N^{-1}\max|\psi'|)^{k+1}.
    \]
    In the case of $\DS(K)=RT_k(K)$ and $\sigma\geq k+3/2$ we obtain
    \[
      \norm{\Div\veta}{L^2(\Omega)}\lesssim \eps^{-1/2}(h+N^{-1}\max|\psi'|)^{k+1},
    \]
    while for $\DS(K)=BDM_k$ and $\sigma\geq k+1/2$ we have
    \[
      \norm{\Div\veta}{L^2(\Omega)}\lesssim \eps^{-1/2}(h+N^{-1}\max|\psi'|)^{k}.
    \]
  \end{lemma}
  \begin{proof}
    Using the solution decomposition \eqref{eq:solDecU2} and the anisotropic interpolation error estimate \eqref{eq:aniso:RT:1}
    we obtain
    \[
      \norm{\vI_2\vs-\vs}{L^2(\Omega)}
        = \norm{\vJ\vs-\vs}{L^2(\Omega)}
        \lesssim (h+N^{-1})^{k+1}\norm{\vs}{H^{k+1}(\Omega)}
        \lesssim \eps(h+N^{-1})^{k+1}.
    \]
    For the boundary layer terms we use the special structure of $\vI_2$ and estimate differently on the subdomains of $\Omega$.
    We show the procedure for $\vw^1$, the estimates of the other terms follow similarly.
    In $\Omega\setminus\Omega_1$ the interpolant is zero and we get
    \[
      \norm{\vI_2\vw^1-\vw^1}{L^2(\Omega\setminus\Omega_1)}
        = \norm{\vw^1}{L^2(\Omega\setminus\Omega_1)}
        \lesssim \eps^{1/2} N^{-\sigma}.
    \]
    In $\Omega_1$ we have
    \[
      \norm{\vI_2\vw^1-\vw^1}{L^2(\Omega_1)}
        \leq \norm{\vJ\vw^1-\vw^1}{L^2(\Omega_1)}+\norm{\vP^1\vw^1}{L^2(\Omega_1\setminus\Omega_1^*)},
    \]
    where $\vP^1:=\hat\vJ^1-\vJ$.
    For the first term we obtain using \eqref{eq:aniso:RT:1} and \eqref{eq:hbound}
    \begin{align*}
      \norm{\vJ\vw^1-\vw^1}{L^2(\Omega_1)}^2
        &=\sum_{K\subset\Omega_1}\sum_{\ell=0}^{k+1}\eps^{2(k+1-\ell)}(N^{-1}\max|\psi'|)^{2(k+1-\ell)}N^{-2\ell}\norm{\e^{\frac{(k+1-\ell)x}{\sigma\eps}}\eps^{-(k+1-\ell)}\e^{-\frac{x}{\eps}}}{L^2(K)}^2\\
        &\lesssim (N^{-1}\max|\psi'|)^{2(k+1)}\norm{\e^{\frac{(k+1-\sigma)x}{\sigma\eps}}}{L^2(\Omega_1)}^2\\
        &\lesssim (N^{-1}\max|\psi'|)^{2(k+1)}\eps,
    \end{align*}
    due to $\sigma>k+1$. In the remaining ply of elements the operator $\vP^1$ in the Raviart-Thomas case is given by
    \begin{align*}
      \int_{F}\vP^1 \vw^1\cdot\vn\cdot q&=0\text{ for all faces }F\subset\partial K\setminus\Gamma_1,\,\forall q\in\PS_k(F),\\
      \int_{F}\vP^1\vw^1\cdot\vn\cdot q&=\int_{F}\vw^1\cdot\vn\cdot q\text{ for all faces }F\subset\partial K\cap\Gamma_1,\,\forall q\in\PS_k(F),\\
      \int_K(\vP^1\vw^1)\cdot \vq&=0,\,\forall\vq\in\QS_{k-1,k}(K)\times\QS_{k,k-1}(K),
    \end{align*}
    and similarly for the other finite elements. Thus $\vP^1\vw^1$ depends only on $\vw^1\cdot\vn|_{\Gamma_1}$. 
    With $\vP^1$ on $\Gamma_1$ being defined by weighted integrals, we have  
    \begin{gather}\label{eq:P1w1}
      \norm{\vP^1\vw^1}{L^2(\Omega_1\setminus\Omega_1^*)}
        \lesssim \meas(\Omega_1\setminus\Omega_1^*)^{1/2}\norm{\vw^1\cdot\vn}{L^\infty(\Gamma_1)}
        \lesssim h_{N/4}^{1/2}N^{-\sigma}
        \lesssim \eps^{1/2}N^{-(k+1)}.
    \end{gather}
    Applying the same techniques to the other boundary and corner layer terms, and collecting the result finishes the first part of the proof.

    For the divergence we can apply the same techniques with the difference of applying 
    \eqref{eq:aniso:RT:2} instead of \eqref{eq:aniso:RT:1}. We obtain for $\sigma> k+1$ and $\DS(K)=RT_k(K)$ 
    \begin{align*}
      \norm{\Div(\vI_2\vs-\vs)}{L^2(\Omega)}
        &\lesssim (h+N^{-1})^{k+1}\norm{\Div\vs}{H^{k+1}(\Omega)}
        \lesssim \eps(h+N^{-1})^{k+1},\\
      \norm{\Div(\vI_2\vw^1-\vw^1)}{L^2(\Omega\setminus\Omega_1)}
        &= \norm{\Div\vw^1}{L^2(\Omega\setminus\Omega_1)}
        \lesssim \eps^{-1/2} N^{-\sigma},\\
      \norm{\Div(\vJ\vw^1-\vw^1)}{L^2(\Omega_1)}
        &\lesssim \eps^{-1}(N^{-1}\max|\psi'|)^{k+1}\norm{\e^{\frac{(k+1-\sigma)x}{\sigma\eps}}}{L^2(\Omega_1)}\\
        &\lesssim \eps^{-1/2}(N^{-1}\max|\psi'|)^{k+1}.
    \end{align*}
    The last term to estimate is the error on the ply of elements in $\Omega_1\setminus\Omega_1^*$.
    A closer inspection of $\vP^1\vw^1$ reveals $(\vP^1\vw^1)_2=0$.
    Thus, an inverse inequality followed by \eqref{eq:P1w1} yields
    \begin{align*}
      \norm{\Div(\vP^1\vw^1)}{L^2(\Omega_1\setminus\Omega_1^*)}
        &\lesssim h_{N/4}^{-1} \norm{(\vP^1\vw^1)_1}{L^2(\Omega_1\setminus\Omega_1^*)}
        \lesssim h_{N/4}^{-1/2} \norm{(\vw^1)_1}{L^\infty(\Gamma_1)}
        \lesssim \eps^{-1/2}N^{1/2}N^{-\sigma},
    \end{align*}
    where $h_{N/4}\geq h_{min}\geq \eps N^{-1}$ holds due to the assumptions on $\phi$. 
    The analysis for the other terms of the decomposition follows the same lines.
    
    For $\DS(K)=BDM_k(K)$
    the same analysis can be done, only replacing the convergence orders by $k$ for $\sigma\geq k+1/2$.
  \end{proof}

  \begin{lemma}\label{lem:est:eta1div}
    Assuming $h\eps\lesssim N^{-2}$
    and $\sigma\geq k+1$, it holds for any $V=(v,\vv)\in\U_N$
    \[
      |\scp{\eta,\Div\vv}|\lesssim\eps^{-1/2}(h+N^{-1}\max|\psi'|)^{k+1}(\ln N)^{1/2}\norm{\vv}{L^2(\Omega)}.
    \]
  \end{lemma}
  \begin{proof}
    Let $c_K:=\frac{1}{\meas(K)}\int_K c\geq c_0$ be a piecewise constant approximation of $c$. Now
    \begin{align*}
      \scp{\eta,\Div\vv}
        &= \sum_{K\in T_N}\scp{\eta,\Div\vv}_K
         = \sum_{K\in T_N}\frac{1}{c_K}\scp{(c_K-c)\eta,\Div\vv}_K,
    \end{align*}
    due to the weighted $L^2$-projection and $\Div\vv|_K\in \QS_k(K)$. It holds
    \[
      \norm{c_K-c}{L^\infty(K)}\lesssim (h_x+h_y)\norm{c}{W^{1,\infty}(K)}.
    \]
    Thus we obtain, using $h_x$ and $h_y$ as abbreviations for the dimensions of $K$, an inverse inequality
    and \eqref{eq:aniso:L2} for any $v\in H^{k+1}(K)$
    \begin{align*}
      |&\scp{(c_K-c)(v-\I_1 v),\pt_x \vv_1}_K|\\
        &\lesssim \left(1+\frac{h_y}{h_x}\right)\norm{v-\I_1 v}{L^2(K)}\norm{\vv}{L^2(K)}\\
        &\lesssim \left((h_x+h_y)\norm{h_x^{k}\pt_x^{k+1}v}{L^2(K)}+\left(1+\frac{h_y}{h_x}\right)\norm{h_y^{k+1}\pt_y^{k+1}v}{L^2(K)}\right)\norm{\vv}{L^2(K)}
    \end{align*}
    and similarly for the $y$-derivative of the second component.
    
    Let us start with the smooth part $s$ of the solution decomposition and denote the coarse part of $\Omega$ 
    by $\Omega_c:=\Omega\setminus\bigcup_{i=1}^4\Omega_i$ and the union of corners by $\Omega_{cor}:= \Omega_{12}\cup\Omega_{23}\cup\Omega_{34}\cup\Omega_{41}$.
    Then we obtain
    \begin{align*}
      |\scp{(c_K-c)(s-\I_1 s),\pt_x \vv_1}|
        \lesssim &\Big((h+N^{-1})^{k+1}\norm{\pt_x^{k+1}s}{L^2(\Omega)}
                                   +N^{-(k+1)}\norm{\pt_y^{k+1}s}{L^2(\Omega_c)}\\
                                  &+\eps^{-1}N^{-(k+1)}\norm{\pt_y^{k+1}s}{L^2((\Omega_1\cup\Omega_3)\setminus\Omega_{cor})}\\
                                  &+hNh^{k+1}\norm{\pt_y^{k+1}s}{L^2((\Omega_2\cup\Omega_4)\setminus\Omega_{cor})}\\
                                  &+ Nh^{k+1}\norm{\pt_y^{k+1}s}{L^2(\Omega_{cor})}
                                  \Big)\norm{\vv}{L^2(\Omega)}\\
        \lesssim &\eps^{-1/2}(h+N^{-1})^{k+1}\left(1+h\eps N^2\right)(\ln N)^{1/2}\norm{\vv}{L^2(\Omega)}\\
        \lesssim &\eps^{-1/2}(h+N^{-1})^{k+1}(\ln N)^{1/2}\norm{\vv}{L^2(\Omega)}
    \end{align*}
    due to $h_{min}\geq\eps^{-1}N$ and the condition on $h\eps$. In the final estimate on $\Omega_{cor}$
    we have used
    \[
      Nh^{k+1}
        \leq hN(h+N^{-1})^k
        = hN^2N^{-1}(h+N^{-1})^k
        \leq hN^2(h+N^{-1})^{k+1}.
    \]
    For $\pt_y \vv_2$ holds a similar estimate due to symmetry.
    Next we look at the boundary layer term $w_1$. We obtain in $\Omega_1$
    \begin{align*}
      |&\scp{(c_K-c)(w_1-\I_1 w_1),\pt_x \vv_1}_{\Omega_1}|\\
        &\lesssim \bigg((h+N^{-1}\max|\psi'|)^{k+1}\eps^{-1}\norm{\e^{\frac{kx}{\sigma\eps}}\e^{-\frac{x}{\eps}}}{L^2(\Omega_1)}+\\&\hspace*{2cm}
                                                   \eps^{-1}N^{-(k+1)}\norm{\e^{-\frac{x}{\eps}}}{L^2(\Omega_1\setminus\Omega_{cor})}+
                                                   Nh^{k+1}\norm{\e^{-\frac{x}{\eps}}}{L^2(\Omega_1\cap\Omega_{cor})}\bigg)\norm{\vv}{L^2(\Omega_1)}\\
        &\lesssim \eps^{-1/2}(h+N^{-1}\max|\psi'|)^{k+1}\left(1+h\eps N^2\right)\norm{\vv}{L^2(\Omega_1)}\\
        &\lesssim \eps^{-1/2}(h+N^{-1}\max|\psi'|)^{k+1}\norm{\vv}{L^2(\Omega_1)},
    \end{align*}
    using again the condition on $\eps h$. Now for the $y$-derivative it holds
    \begin{align*}
      |&\scp{(c_K-c)(w_1-\I_1 w_1),\pt_y \vv_2}_{\Omega_1}|\\
        &\lesssim \bigg((h+N^{-1})^{k+1}\norm{\e^{-\frac{x}{\eps}}}{L^2(\Omega_1)}+\\
                        &\hspace{2cm} (1+hN)(N^{-1}\max|\psi'|)^{k+1}\norm{\e^{\frac{(k+1)x}{\eps}}\e^{-\frac{x}{\eps}}}{L^2(\Omega_1\setminus\Omega_{cor})}+\\
                        &\hspace{2cm} (1+h\eps^{-1}N)(N^{-1}\max|\psi'|)^{k+1}\norm{\e^{\frac{(k+1)x}{\eps}}\e^{-\frac{x}{\eps}}}{L^2(\Omega_1\cap\Omega_{cor})}\bigg)\norm{\vv}{L^2(\Omega_1)}\\
        &\lesssim \eps^{-1/2}(h+N^{-1}\max|\psi'|)^{k+1}(\eps+h\eps N(\ln N)^{1/2}+h\eps^{3/2}N\ln N)\norm{\vv}{L^2(\Omega_1)}\\
        &\lesssim \eps^{-1/2}(h+N^{-1}\max|\psi'|)^{k+1}\norm{\vv}{L^2(\Omega_1)}.
    \end{align*}
    In the remainder of the domain we apply the $L^\infty$-stability and get
    \begin{align*}
      |\scp{(c_K-c)(w_1-\I_1 w_1),\pt_x \vv_1}_{\Omega\setminus\Omega_1}|
        &\lesssim N^{-\sigma}(1+Nh\meas(\Omega\setminus\Omega_1\setminus\Omega_c)^{1/2})\norm{\vv}{L^2(\Omega\setminus\Omega_1)}\\
        &\lesssim \eps^{-1/2}N^{-\sigma}(\eps^{1/2}+h\eps N(\ln N)^{1/2})\norm{\vv}{L^2(\Omega\setminus\Omega_1)}\\
        &\lesssim \eps^{-1/2}N^{-(k+1)}\norm{\vv}{L^2(\Omega\setminus\Omega_1)},\\
      |\scp{(c_K-c)(w_1-\I_1 w_1),\pt_y(\vv)_2}_{\Omega\setminus\Omega_1}|
        &\lesssim N^{-\sigma}(1+\eps^{-1}\meas(\Omega\setminus\Omega_1\setminus\Omega_c)^{1/2})\norm{\vv}{L^2(\Omega\setminus\Omega_1)}\\
        &\lesssim \eps^{-1/2}N^{-(k+1)}(\ln N)^{1/2}\norm{\vv}{L^2(\Omega\setminus\Omega_1)}.
    \end{align*}
    The estimation of the other boundary layer terms and of the corner layer terms is similar. Combining all the 
    individual results proves the assertion.
  \end{proof}
  \begin{lemma}\label{lem:discreteerror}
    For $h\eps\lesssim N^{-2}$ it holds for $\DS(K)=RT_k(K)$ with $\sigma\geq k+3/2$
    \[
      \tnorm{(\xi,\vxi)}_{bal}\lesssim (h+N^{-1}\max|\psi'|)^{k+1}(\ln N)^{1/2}
    \]
    and for $\DS(K)=BDM_k(K)$ with $\sigma\geq k+1$
    \[
      \tnorm{(\xi,\vxi)}_{bal}\lesssim (h+N^{-1}\max|\psi'|)^{k}.
    \]
  \end{lemma}
  \begin{proof}
    Using the inf-sup estimate \eqref{eq:infsup} and the previous lemmas we obtain for $\DS(K)=RT_k(K)$
    \[
      \beta\tnorm{(\xi,\vxi)}
        \leq \sup_{V\in\U_N}\frac{B((\eta,\veta),V)}{\norm{V}{L^2(\Omega)}}
        \lesssim \eps^{1/2}(h+N^{-1}\max|\psi'|)^{k+1}(\ln N)^{1/2},
    \]
    particularly
    \begin{align*}
      \norm{\xi}{L^2(\Omega)}
        &\lesssim \eps^{1/2}(h+N^{-1}\max|\psi'|)^{k+1}(\ln N)^{1/2},\\
      \eps^{-1/2}\norm{\vxi}{L^2(\Omega)}
        &\lesssim (h+N^{-1}\max|\psi'|)^{k+1}(\ln N)^{1/2},\\
      \delta^{1/2}\eps^{-1/2}\norm{\eps\Div \vxi}{L^2(\Omega)}
        &\lesssim (h+N^{-1}\max|\psi'|)^{k+1}(\ln N)^{1/2},
    \end{align*}
    which are the three components of $\tnorm{\xi}_{bal}$. Similar results follow for 
    $\DS(K)=BDM_k(K)$ with the additional
    \[
      (h+N^{-1}\max|\psi'|)(\ln N)^{1/2}\lesssim 1.\qedhere
    \]
  \end{proof}
  \begin{theorem}\label{thm:main}
    For $h\eps\lesssim N^{-2}$ it holds for $\DS(K)=RT_k(K)$ with $\sigma\geq k+3/2$
    \[
      \tnorm{U-U_N}_{bal}\lesssim (h+N^{-1}\max|\psi'|)^{k+1}(\ln N)^{1/2}.
    \]
    and for $\DS(K)=BDM_k(K)$ with $\sigma\geq k+1$
    \[
      \tnorm{U-U_N}_{bal}\lesssim (h+N^{-1}\max|\psi'|)^{k}.
    \]
  \end{theorem}
  \begin{proof}
    With the triangle inequality
    \[
      \tnorm{U-U_N}_{bal}\leq \tnorm{(\xi,\vxi)}_{bal}+\tnorm{(\eta,\veta)}_{bal}
    \]
    and the previous lemmas it only remains to estimate $\norm{\eta}{L^2(\Omega)}$, which can
    be done using the local anisotropic interpolation error estimates \eqref{eq:aniso:L2}
    by standard techniques 
    \[
      \norm{\eta}{L^2(\Omega)}
        \lesssim (h+N^{-1}\max|\psi'|)^{k+1}.\qedhere
    \]
  \end{proof}
  \begin{corollary}
    Under the same conditions as the previous theorem we also have for $\DS(K)=RT_k(K)$
    in the unbalanced norm
    \[
      \tnorm{U-U_N}\lesssim (h+N^{-1}\max|\psi'|)^{k+1}.
    \]
  \end{corollary}
  \begin{remark}
    On a Shishkin mesh we have $h_i=\eps N^{-1}\ln N$ inside the boundary domain.
    Thus the condition on $\eps h$ becomes
    \[
      \eps^2\lesssim \frac{N^{-1}}{\ln N}
      \qmbox{and}
      h+N^{-1}\max|\psi'|
      \lesssim N^{-1}\ln N.
    \]
    On a Bakhvalov S-mesh we have $h\sim\eps$ and the condition becomes
    \[
      \eps\lesssim N^{-1}
      \qmbox{and therefore}
      h+N^{-1}\max|\psi'|
      \lesssim N^{-1}.
    \]
  \end{remark}
  \begin{remark}
    The same analysis can also be conducted for the Arnold-Boffi-Falk element 
    \[
      \DS(K)=ABF_k(K):=\QS_{k+2,k}(K)\times\QS_{k,k+2}(K),
    \]
    see \cite{ABF05}, using $\Div\DS(K)=\QS_{k+1}(K)\setminus\text{span}\{x^{k+1}y^{k+1}\}$
    as discrete space for the first component. Anisotropic interpolation error estimates are given
    in \cite{Fr21}. Although $\norm{\eta}{L^2(\Omega)}$ and $\norm{\Div\veta}{L^2(\Omega)}$ can be estimated with order $k+2$, we obtain 
    only convergence rates of order $k+1$ due to $\norm{\veta}{L^2(\Omega)}\lesssim \eps^{1/2}(h+N^{-1}\max|\psi'|)^{k+1}$.
  \end{remark}

  \section{Numerical experiments}\label{sec:numerics} 
  Let us consider on $\Omega=(0,1)^2$
  \[
    -\eps^2\laplace u+cu=f,
  \]
  where
  \[
    c=1+x^2y^2\e^{xy/2}
    \rarrow c_0=1,\,c_\infty=1+\e^{1/2},\,\delta=\frac{1}{(1+\e^{1/2})^2}
  \]
  and an exact solution
  \[
    u=\left( \cos\left(\frac{\pi x}{2}\right)-\frac{\e^{-x/\eps}-\e^{-1/\eps}}{1-\e^{-1/\eps}}\right)\cdot
      \left( 1-y- \frac{\e^{-y/\eps}-\e^{-1/\eps}}{1-\e^{-1/\eps}}\right)
  \]
  is prescribed, see \cite{RSch15,AMM19} for $c=1$. The solution has only boundary layers at $x=0$ and $y=0$, and a corner layer at $(0,0)$.
  Therefore, we modify our mesh accordingly. For our experiments we will always use Bakhvalov-S-meshes.
 
  All computations were done in $\mathbb{SOFE}$, a finite-element framework in Matlab and Octave, 
  see \texttt{github.com/SOFE-Developers/SOFE}.
  
  Let us start the numerical investigation by looking at the dependence on $\eps$. For that we fix
  $N=16$ and use $RT_1$-elements, and vary $\eps\in\{10^{-3},\,10^{-4},\,10^{-5},\,10^{-6}\}$. We obtain the numbers in Table~\ref{tab:uniform}.
  \begin{table}[tbp]
    \caption{Errors in various norms for varying values of $\eps$ and fixed $N$, $RT_1$.\label{tab:uniform}}
    \begin{center}
      \begin{tabular}{cccccc}
        $\eps$ & $\norm{u-u_h}{L^2(\Omega)}$
              & $\norm{\vu-\vu_h}{L^2(\Omega)}$
              & $\norm{\Div(\vu-\vu_h)}{L^2(\Omega)}$
              & $\tnorm{U-U_h}$
              & $\tnorm{U-U_h}_{bal}$\\
        \hline
   $10^{-3}$ &  5.904e-04 & 4.492e-05 & 1.445e-01 & 5.946e-04 & 2.311e-03\\
   $10^{-4}$ &  5.867e-04 & 1.401e-05 & 4.591e-01 & 5.871e-04 & 2.304e-03\\
   $10^{-5}$ &  5.863e-04 & 4.421e-06 & 1.452e-00 & 5.863e-04 & 2.303e-03\\
   $10^{-6}$ &  5.863e-04 & 1.398e-06 & 4.593e-00 & 5.863e-04 & 2.303e-03
      \end{tabular}
    \end{center}
  \end{table}
  We observe $\norm{u-u_h}{L^2(\Omega)}$ to be independent of $\eps$, as expected, while
  $\norm{\vu-\vu_h}{L^2(\Omega)}\lesssim \eps^{1/2}$ and 
  $\norm{\Div(\vu-\vu_h)}{L^2(\Omega)}\lesssim \eps^{-1/2}$, also both as expected.
  Consequently, $\tnorm{U-U_h}$ stays independent of $\eps$, due to the dominating effect of 
  $\norm{u-u_h}{L^2(\Omega)}$, and the larger balanced norm $\tnorm{U-U_h}_{bal}$ is independent too due to the correct 
  weighting of the other two norms.
  
  Now let us come to the convergence orders. For that purpose we fix $\eps=10^{-4}$ and vary for 
  different values of $k$ the number $N$ of cells per dimension. We start with Raviart-Thomas elements
  and obtain the results of Table~\ref{tab:RT}.
  \begin{table}[tbp]
    \caption{Errors $\tnorm{U-U_h}_{bal}$ for fixed $\eps=10^{-4}$ in the Raviart-Thomas case.\label{tab:RT}}
    \begin{center}
      \begin{tabular}{ccccccc}
        $N$& \multicolumn{2}{c}{$k=1$}
          & \multicolumn{2}{c}{$k=2$}
          & \multicolumn{2}{c}{$k=3$}\\
        \hline
     8 & 9.250e-03 &      & 1.059e-03 &      & 1.640e-04 &     \\
    16 & 2.304e-03 & 2.01 & 1.518e-04 & 2.80 & 1.238e-05 & 3.73\\
    32 & 5.679e-04 & 2.02 & 2.023e-05 & 2.91 & 8.433e-07 & 3.88\\
    64 & 1.402e-04 & 2.02 & 2.603e-06 & 2.96 & 5.480e-08 & 3.94\\
   128 & 3.479e-05 & 2.01 & 3.298e-07 & 2.98 & 3.488e-09 & 3.97\\
   256 & 8.657e-06 & 2.01 & 4.148e-08 & 2.99 & 2.198e-10 & 3.99\\
   512 & 2.158e-06 & 2.00
      \end{tabular}
    \end{center}
  \end{table}
  Here along with the computed errors also the estimated rates of convergence are given
  and they are close to the expected rates of $k+1$ for the balanced norm. 

  In the case of Brezzi-Douglas-Marini elements we get Table~\ref{tab:BDM}.
%
  \begin{table}[tbp]
    \caption{Errors $\tnorm{U-U_h}_{bal}$ for fixed $\eps=10^{-4}$ in the Brezzi-Douglas-Marini case.\label{tab:BDM}}
    \begin{center}
      \begin{tabular}{ccccccc}
        $N$& \multicolumn{2}{c}{$k=1$}
           & \multicolumn{2}{c}{$k=2$}
           & \multicolumn{2}{c}{$k=3$}\\
        \hline
          8 & 8.406e-01 &      & 1.520e-02 &      & 2.700e-03 &     \\
         16 & 6.451e-01 & 0.38 & 3.420e-03 & 2.15 & 3.624e-04 & 2.90\\
         32 & 3.519e-01 & 0.87 & 8.492e-04 & 2.01 & 4.606e-05 & 2.98\\
         64 & 1.796e-01 & 0.97 & 2.144e-04 & 1.99 & 5.553e-06 & 3.05\\
        128 & 7.124e-02 & 1.33 & 5.403e-05 & 1.99 & 6.631e-07 & 3.07\\
        256 & 2.186e-02 & 1.70 & 1.356e-05 & 1.99 & 8.129e-08 & 3.03\\
        512 & 5.906e-03 & 1.89 & 3.390e-06 & 2.00 & 1.027e-08 & 2.98\\
       1024 & 1.566e-03 & 1.92 
      \end{tabular}                                                        
    \end{center}                                                           
\end{table}                                                                
  As expected we only see rates of $k$ in the balanced version with slightly better results for the lowest order case.
  The reason for this behaviour lies in the components of the balanced norms, where the faster converging ones
  dominate for smaller values of $N$ the balanced norm. A closer look reveals $\norm{u-u_h}{L^2(\Omega)}$ and
  $\norm{\Div(\vu-\vu_h)}{L^2(\Omega)}$ only to be convergent with order 1, see Table~\ref{tab:BDM1}.
  \begin{table}[tbp]
    \caption{Errors is various norms for fixed $\eps=10^{-4}$ in the $BDM_1$-case.\label{tab:BDM1}}
    \begin{center}
      \begin{tabular}{ccccccc}
        $N$ & \multicolumn{2}{c}{$\norm{u-u_h}{L^2(\Omega)}$}
            & \multicolumn{2}{c}{$\eps^{-1/2}\norm{\vu-\vu_h}{L^2(\Omega)}$} 
            & \multicolumn{2}{c}{$(\delta\eps)^{1/2}\norm{\Div(\vu-\vu_h)}{L^2(\Omega)}$}\\
        \hline
     8 & 2.786e-03 &      &  8.386e-01 &      & 5.716e-02 &     \\
    16 & 9.940e-04 & 1.49 &  6.444e-01 & 0.38 & 3.034e-02 & 0.91 \\
    32 & 4.369e-04 & 1.19 &  3.516e-01 & 0.87 & 1.554e-02 & 0.96 \\
    64 & 2.114e-04 & 1.05 &  1.794e-01 & 0.97 & 7.860e-03 & 0.98 \\
   128 & 1.051e-04 & 1.01 &  7.113e-02 & 1.33 & 3.953e-03 & 0.99 \\
   256 & 5.254e-05 & 1.00 &  2.177e-02 & 1.71 & 1.982e-03 & 1.00 \\
   512 & 2.629e-05 & 1.00 &  5.822e-03 & 1.90 & 9.926e-04 & 1.00 \\
  1024 & 1.315e-05 & 1.00 &  1.485e-03 & 1.97 & 4.965e-04 & 1.00
      \end{tabular}
    \end{center}
  \end{table}
%

  \bibliographystyle{plain}
  \bibliography{lit}

\end{document}